\documentclass{amsart}
\usepackage{amsmath, amssymb, mathrsfs, verbatim, multirow}
\usepackage[all]{xy}
\usepackage{hyperref}
\usepackage{caption}
\usepackage{imakeidx}
\usepackage{tikz}
\usepackage{tkz-tab}
\usepackage{subcaption} 
\usepackage{float}

\newtheorem{Teo}{Theorem}[section]
\newtheorem{Def}[Teo]{Definition}
\newtheorem{Prop}[Teo]{Proposition}

\newtheorem{Obs}[Teo]{Remark}
\newtheorem{Lema}[Teo]{Lemma}

\newtheorem{Cor}[Teo]{Corollary}

\newtheorem{Que}[Teo]{Question}
\newtheorem{Exa}[Teo]{Example}

\newcommand{\R}{\mathbb{R}}
\newcommand{\Z}{\mathbb{Z}}
\newcommand{\N}{\mathbb{N}}

\newcommand{\Llr}{\Longleftrightarrow}
\newcommand{\lra}{\longrightarrow}
\newcommand{\Lra}{\Longrightarrow}
\newcommand{\VR}{\mathcal{O}}

\title[K\"ahler differentials]{K\"ahler differentials, pure extensions and minimal key polynomials}
\author{Josnei Novacoski}
\address{Departamento de Matem\'{a}tica,  Universidade Federal de S\~ao Carlos, Rod. Washington Luís, 235, 13565--905, S\~ao Carlos -SP, Brazil and Instytut Matematyki, Uniwersytet \'Sl\c aski w Katowicach, ul. Bankowa 14, 40-007, Katowice, Poland}
\email{josnei@ufscar.br}
\author{Mark Spivakovsky}
\address{CNRS UMR 5219 and Institut de Mathématiques de Toulouse, 118, rte de Narbonne, 31062 Toulouse cedex 9, France and Instituto de Matem\'aticas (Cuernavaca) LaSol, UMI CNRS 2001, UNAM, Av. Universidad s/n. Col. Lomas de Chamilpa, 62210, Cuernavaca, Morelos, M\'exico}
\email{mark.spivakovsky@math.univ-toulouse.fr}

\thanks{During the realization of this project the first author was supported by a grant from Funda\c{c}\~ao de Amparo \`a Pesquisa do Estado de S\~ao Paulo (process number 2017/17835-9).}
\begin{document}

\begin{abstract}
The main object of study in this paper is the module $\Omega$ of K\"ahler differentials of an extension of valuation rings. We show that in the case of pure extensions $\Omega$ has a very good description. Namely, it is isomorphic to the quotient of two modules defined by final segments in the value group of the valuation. This allows us to describe the annihilator of $\Omega$ and to give criteria for
$\Omega$ to be finitely generated and to be finitely presented. We also discuss how to relate the module of K\"ahler differentials of pure extensions to minimal key polynomials.  
\end{abstract}
\maketitle
\section{Introduction}
Our motivation for studying the module of K\"ahler differentials of extensions of valuation rings comes from \cite{CK} and \cite{CKR}. The main goal of that work is to present a simple characterization of \textit{deeply ramified fields}. Their first step is to understand the module of K\"ahler differentials of an extension of valuation rings whose corresponding quotient field extension is Galois and of prime degree. Using this, the authors give a criterion for when the module of K\"ahler differentials of a valuation ring extension is zero, assuming that the corresponding extension of valued fields is Galois and finite (not necessarily of prime degree). Using this criterion, one can understand better when a given valued field is deeply ramified.

This work is the first step of the program whose goal is to understand the module $\Omega$ of K\"ahler differentials of an extension of valuation rings in the most general possible case. For instance, if $\Omega\neq (0)$, then it is interesting to have an explicit characterization for its structure. In order to understand the module of K\"ahler differentials of a ring extension $B|A$, the first step is to present a set of generators for it. If $B=\VR_L$ and $A=\VR_K$ are the valuation rings of an unramified simple valued field extension $(L|K,v)$, then \textit{complete sets} (defined in Section \ref{secpreliminar} below) naturally give rise to sets of generators of $\VR_L$ over $\VR_K$. However, it is important to have a set of generators whose relations are as simple as possible. For this purpose, complete sets of \textit{key polynomials} are very useful.  

Let $(L|K,v)$ be a simple algebraic extension of valued fields and $\Gamma=vL$ --- the value  group of $(L,v)$. For a subset $S$ of $\Gamma$, we denote by $I_S$ the $\VR_L$-submodule of $L$ defined by $S$, that is, 
\[
I_S=\{b\in L\mid \mbox{ there exists }s\in S\mbox{ such that }vb \geq s\}.
\]
For any \textit{sequence of key polynomials} $\textbf{Q}$ for $(L|K,v)$ (see Section \ref{secpreliminar} below), we define two final segments $\alpha$ and $\beta$ of $\Gamma$. The main idea behind this work is that, under certain conditions, the $\VR_L$-module
$\Omega$ of K\"ahler differentials of the extension $\VR_L|\VR_K$ is isomorphic to $I_\alpha/I_\beta$. This allows us to present an explicit formula for the annihilator of $\Omega$ (Proposition \ref{propsannihilator}). We also present conditions for $\Omega$ to be finitely generated and to be finitely presented (Proposition \ref{finitelygenerated} and Corollary \ref{finitelypresentes}).

A natural situation when this idea can be applied is when $(L|K,v)$ is a \textit{pure extension} (see Definition \ref{definintiaopure}). Such an extension admits a complete sequence of key polynomials of degree one. Hence, we can present simple formulas for $\alpha$ and $\beta$ (\eqref{formulaalpha} and \eqref{formulabeta}, respectively) and therefore also for $\Omega$. This allows us to present a simple characterization for the annihilator of $\Omega$ (see \eqref{formulaannilhi}). Also, it is easy to show that, in this case, $\Omega$ is not finitely generated and, in particular, not finitely presented.  

By a {\bf plateau} of key polynomials we mean the set of all the key polynomials of a constant degree. In Section \ref{pureexteexts}, we extend the theory of pure extensions (as presented in \cite{CKR}). In \cite{CKR}, only immediate extensions are considered. These can be \textit{pure defect} or \textit{branched pure} extensions. A pure defect extension can be understood as an extension consisting of only one plateau of key polynomials and this plateau encodes the defect of the extension. Based on this, we consider the cases of \textit{purely inertial} and \textit{purely ramified} extensions. Analogously to the pure defect case, these extensions have only one plateau of key polynomials and this plateau encodes the inertia and ramification indices, respectively.  

Finally, we apply the theory of \textit{minimal key polynomials} to pure extensions (see Section \ref{minkeypolyno}). Minimal limit key polynomials are the main objects of study in \cite{NN2023}. For each plateau of key polynomials $\Psi$ we define a finite subset $B$ of $\N$ that characterizes the minimal key polynomials for
$\Psi$. If the extension is pure, then it admits only the plateau consisting of linear key polynomials. We show that for any pure defect, branched pure or purely inertial extension $(L|K,v)$, with $p={\rm char}(Kv)>0$, we have
\begin{equation}\label{omegaminim}
\Omega=(0)\Llr p\nmid B.
\end{equation}
Let us give a more precise description of this result.

If $(L|K,v)$ is of pure defect, then, by the results of \cite{NN2023}, $p\nmid B$ if and only if $1\in B$. Then the above result follows from the fact (Proposition \ref{Propbunitinhamans}) that, under these assumptions, $1\in B$ if and only if
$\Omega=(0)$.

If $(L|K,v)$ is branched pure, then each of the conditions in \eqref{omegaminim} is satisfied if and only if the extension is defectless. This is a consequence of Proposition \ref{Propbunitinhamans} and of an explicit description of $B$ presented in \cite{NN2023}.

Finally, in the purely inertial case, we show that each of the conditions in \eqref{omegaminim} is satisfied if and only if the residue field extension is separable.

In the purely ramified case, we cannot obtain a set of generators for the $\VR_K$-algebra $\VR_L$ from sequences of key polynomials (unless we allow monomials in the key polynomials with integer, possibly negative exponents). However, we adapt previous results to provide a similar description of $\Omega$ (see Section \ref{seclmpureramifi}).

To summarize, what we consider to be the main new results of the present paper compared to the earlier works on the subject are the following.
\begin{description}
\item[(i)] We explicitly compute the $\VR_L$-module $\Omega$ for all pure extensions of arbitrary degree as opposed to the defect
Artin--Schreier and Kummer extensions of degree $p$ as in \cite{CKR}. Our notion of pure extension is more general than that originally defined in \cite{CKR} as it also includes purely ramified and purely inertial extensions. The original pure extensions are called pure immediate extensions in our terminology.
\item[(ii)] Our description of $\Omega$ allows us to explicitly describe the annihilator ${\rm ann}(\Omega)$ for all pure extensions (not necessarily Galois) and, in particular, to characterize the situation when
\begin{equation}
\Omega=(0).\label{eq:Omegais(0)}
\end{equation}
In \cite{CK}, the equality \eqref{eq:Omegais(0)} is characterized for arbitrary Galois extensions (not necessarily pure).

\item[(iii)] We use the notion of minimal key polynomials as in \cite{NN2023} to characterize, in many cases, when the equality \eqref{eq:Omegais(0)} is satisfied.
\end{description}
\medskip
\noindent\textbf{Acknowledgement.} This paper was inspired by discussions with Franz-Viktor Kuhlmann and Hagen Knaf.
\section{Preliminaries}\label{secpreliminar}
\medskip

\noindent{\bf Notation.} In this paper we will use the letter $v$ to denote a valuation of $K$ or its extensions to algebraic extensions of $K$. For valuations of $K[x]$, we will use the letters $\nu$ and $\mu$.

For a valued field $(K,v)$ we will denote by $\VR_K$ the valuation ring, by $vK$ the value group and by $Kv$ the residue field of $v$. Also, for $b\in K$ we denote by $vb$ or $v(b)$ the value of $b$ in $vK$. If $b\in \VR_K$, then we denote by $bv$ the residue of $b$ in $Kv$.

For a finite valued field extension $(L|K,v)$ we will denote by $e(L|K,v)$ and $f(L|K,v)$ the ramification and inertial indices of
$(L|K,v)$, respectively:
\[
e(L|K,v)=(vL:vK)\mbox{ and }f(L|K,v)=[Lv:Kv].
\]
We will denote by $d(L|K,v)$ the defect of $(L|K,v)$. This can be defined as follows. Let $L^h$ and $K^h$ be the \textit{henselizations} of $L$ and $K$, respectively, determined by a fixed extension of $v$ to $\overline{K}$. Then
\[
d(L|K,v):=\frac{[L^h:K^h]}{(vL:vK)\cdot [Lv:Kv]}.
\]

Throughout this paper, when talking about a valued field extension $(L|K,v)$, we will use the symbol $\Omega$ to denote the
$\VR_L$-module of K\"ahler differentials of the extension $\VR_L|\VR_K$. 

\medskip
\subsection{Key polynomials}
The problem of understanding all the possible extensions of a valuation of $K$ to $K[x]$ has been extensively studied since the beginning of the past century. Many objects have been introduced to ``parametrize" these valuations. Here, we will focus on one of them: key polynomials.

Even key polynomials appear in the literature in different ways. For instance, there are \textit{Mac Lane--Vaquié key polynomials} and \textit{abstract key polynomials}. The relation between these two objects is well-known and we will briefly discuss it here.

Roughly speaking, for a valuation $\nu$ on $K[x]$, Mac Lane--Vaquié key polynomials tell the ``future" of $\nu$ and abstract key polynomials tell its ``past". In mathematical terms, a Mac Lane--Vaquié key polynomial for a valuation $\nu$ on $K[x]$ allows us to construct another valuation $\mu$ such that
\[
\nu(f)\leq \mu(f)\mbox{ for every }f\in K[x].
\]
On the other hand, an abstract key polynomial for $\nu$ allows us to construct another valuation $\mu$ such that
\[
\nu(f)\geq \mu(f)\mbox{ for every }f\in K[x].
\]

This means that for treating a given problem, one of these objects may be more suitable than the other. For the goal of this paper, we believe that abstract key polynomials are more suitable. Hence, we will just refer to them as \textit{key polynomials}.

For a valuation $\nu$ on $K[x]$ we fix an extension of it to $\overline{K}[x]$ and denote it again by $\nu$. For $f\in K[x]\setminus K$ we define
\[
\delta(f)=\max\{\nu(x-a)\mid f(a)=0\}.
\]
A monic polynomial $Q$ is a \textbf{key polynomial} for $\nu$ if for every $f\in K[x]$ we have
\[
\delta(f)\ge\delta(Q)\Lra\deg(f)\ge\deg(Q).
\]
For any polynomials $f\in K[x]$ and $q\in K[x]\setminus K$, there exist uniquely determined $f_0,\ldots, f_r\in K[x]$ with
$\deg(f_\ell)<\deg(q)$ for every $\ell$, $0\leq \ell\leq r$, such that
\begin{equation}
f=f_0+f_1q+\ldots+f_rq^r.\label{eq:qexpansion}
\end{equation}
This expression is called the \textbf{$q$-expansion of $f$}. For every monic polynomial
\[
q\in K[x]\setminus K
\]
we define the \textbf{truncation of $\nu$ at $q$} by
\[
\nu_q(f)=\min_{0\leq\ell\leq r}\left\{\nu\left(f_\ell q^\ell\right)\right\}
\]
where the notation is as in \eqref{eq:qexpansion}. We allow the possibility
\begin{equation}
\nu(q)=\infty.\label{eq:nuqinfty}
\end{equation}
If $q$ the has smallest degree among the polynomials satisfying \eqref{eq:nuqinfty}, then $\nu_q(f)=\infty$ if and only if $f_0=0$.

If $Q$ is a key polynomial for $\nu$, then by \cite[Proposition 2.6]{NS2018} $\nu_Q$ is a valuation with
\[
\nu(f)\geq \nu_Q(f)\mbox{ for every }f\in K[x].
\]

A subset $\textbf{Q}$ of $K[x]$ is called a {\bf complete set} for a valuation $\nu$ on $K[x]$, if for every $f\in K[x]$ there exists $q\in \textbf Q$ such that
\[
\deg(q)\leq \deg(f)\mbox{ and }\nu_q(f)=\nu(f).
\]

We consider a valued field extension $(L|K,v)$. We assume that this extension is simple and algebraic and write it as $L=K(\eta)$ for some $\eta\in L$. Consider the valuation $\nu$ on $K[x]$ with non-trivial support defined by $v$ and $\eta$, that is,
\[
\nu(f):=v(f(\eta))\mbox{ for every }f\in K[x].
\]
The philosophy is that a complete set for $\nu$ should induce a set of generators for the extension $\VR_L|\VR_K$ of valuation rings.  This is true whenever $(L|K,v)$ is unramified, that is, $vL=vK$ (see Proposition \ref{Corfinitelgenalg} below).

In order to study the module of K\"ahler differentials, it is necessary to know the generators of the extension, but also the relations between these generators. Because of this, it is important to have a complete set with relations as simple as possible. For this purpose, complete sets formed by key polynomials are very handy. 
\begin{Obs}
We will use the term ``complete sequence" instead of ``complete set" when $\textbf{Q}=\{Q_i\}_{i\in I}$ with $I$ well-ordered and the map $i\mapsto Q_i$ is order preserving with respect to $\delta$ ($i<j\Lra \delta(Q_i)<\delta(Q_j)$).
\end{Obs}

\begin{Teo}\cite[Theorem 1.1]{NS2018}\label{them2.2}
Every valuation $\nu$ on $K[x]$ admits a complete sequence of key polynomials. 
\end{Teo}
\begin{Obs}\begin{enumerate}
\item It is not hard to prove from definitions that the set of all key polynomials is complete. A sequence as in Theorem \ref{them2.2} can be constructed as follows. For every $m\in \N$ consider the set $\Psi_m$ of all the key polynomials of degree $m$. For each $m\in\N$ such that $\Psi_m\neq \emptyset$ we define $\textbf{Q}_m$ as follows. If $\nu(\Psi_m)$ has a last element
$\nu(Q)$ with respect to $\delta$, then put $\textbf{Q}_m=\{Q\}$. If $\nu(\Psi_m)$ does not have a last element, take any sequence $\textbf Q_m$ whose images form an a cofinal subset of $\nu(\Psi_m)$. Then
\[
\textbf{Q}:=\bigcup_{\substack{m\in\N\\\Psi_m\neq\emptyset}}\textbf{Q}_m
\]
is the desired sequence of key polynomials.
\item Assume that
\[
\mbox{\rm supp}( \nu):=\{f\in K[x]\mid \nu(f)=\infty\}
\]
is a principal ideal $(g)$ of $K[x]$ (this is the only situation considered in this paper). Then the complete sequence $\bf Q$ of key polynomials can be chosen so that $g\in\bf Q$ and $\nu(Q)<\infty$ for all $Q\in{\bf Q}\setminus\{g\}$. In particular, $g$ is the last element of $\bf Q$. From now on, we will assume that $\bf Q$ has this property without mentioning it explicitly. 
\end{enumerate}
\end{Obs}
\begin{Obs}
Every linear monic polynomial is a key polynomial (in particular, $\Psi_1\neq \emptyset$). If $\nu(\Psi_1)$ does not have a last element, then the cofinal families $\textbf Q_1$ in $\Psi_1$ are in bijection with the \textit{pseudo-convergent sequences} on $K$ induced by $\nu$. This bijection can be presented explicitly:
\[
\textbf Q_1=\{x-a_i\}_{i\in I}\longmapsto \{a_i\}_{i\in I}.
\]
\end{Obs}

\section{The generation of an extension of valuation rings}
  
Let $(K,v)$ be a valued field and take a valuation $\mu$ on $K[x]$, extending $v$. For simplicity of notation, in this section we denote
\[
\Gamma_v=vK\mbox{ and }\Gamma_\mu= \mu(K[x]).
\]
Assume that $e(\mu/v):=(\Gamma_\mu:\Gamma_v)=1$.

\begin{Def}
Two sets $\textbf{Q}, \textbf{Q}'\subset K[x]$ are called $K$-proportional if there exists a bijection $\varphi:\textbf Q \lra \textbf{Q}'$ such that for every $q\in \textbf{Q}$ there exists $a\in K\setminus\{0\}$ such that
$\varphi(q)=aq$.
\end{Def}
\begin{Lema}\label{lemaparasetpol}
If $\textbf{Q}\subseteq K[x]$ is a complete set for $\mu$, then any $K$-proportional set is also a complete set for $\mu$.
\end{Lema}
\begin{proof}
For $f,q\in K[x]$ and $a\in K\setminus\{0\}$, if
\[
f=f_0+f_1q+\ldots+f_nq^n
\]
is the $q$-expansion of $f$, then
\[
f=f_0+\frac{f_1}{a}\cdot(aq)+\ldots+\frac{f_n}{a^n}\cdot(aq)^n
\]
is the $aq$-expansion of $f$. If $\mu_q(f)=\mu(f)$, then $\mu_{aq}(f)=\mu(f)$ for every $a\in K$. Moreover, $\deg(q)=\deg(aq)$. The result follows immediately.
\end{proof}

\begin{Lema}\label{propostiset}
If $e(\mu/v)=1$, then for every complete set $\textbf{Q}$ of $\mu$ there exists a complete set $\tilde{\textbf{Q}}$ for $\mu$,
$K$-proportional to $\textbf{Q}$, such that $\mu(q)=0$ for every $q\in\tilde{\textbf{Q}}\setminus\{g\}$.
\end{Lema}
\begin{proof}
Since $e(\mu/v)=1$, for every $q\in K[x]\setminus(g)$ there exists $a\in K$ such that 
\[
v(a^{-1})=\mu(q).
\]
Hence $\mu(aq)=0$ and we can apply the previous lemma. 
\end{proof}

For a subset $\textbf{Q}$ of $K[x]$, we will denote by $\N_0^\textbf{Q}$ the set of mappings $\lambda:\textbf{Q}\lra\N_0$ such that $\lambda(q)=0$ for all but finitely many $q\in \textbf Q$ (here $\N_0$ denotes the set of the non-negative integers). For each $\lambda\in\N_0^\textbf{Q}$ we denote
\[
\textbf{Q}^\lambda=\prod_{q\in \textbf Q}q^{\lambda(q)}.
\] 

\begin{Teo}\label{Theoremkeypol}
Take a complete set $\textbf{Q}\subseteq K[x]$ for $\mu$ such that $\mu(q)=0$ for every $q\in \textbf{Q}\setminus\{g\}$. For $f\in K[x]\setminus(g)$, if $\mu(f)\geq 0$, then there exist $\lambda_1,\ldots,\lambda_s\in\N^\textbf{Q}$ and $a_1,\ldots,a_s\in \VR_K$ such that
\[
f=\sum_{\ell=1}^s a_\ell\textbf{Q}^{\lambda_\ell}
\]
and
\begin{equation}
\mu(f)=\min\limits_{1\leq \ell \leq s}v(a_\ell).\label{eq:minimalvalue}
\end{equation}
\end{Teo}
\begin{proof}
We proceed by induction on the degree of $f$. If $\deg(f)=1$, then by our assumption there exists $q\in\textbf{Q}$ of degree one (in this case $f=aq+b$ for $a,b\in K$), such that
\[
\mu(f)=\min\{\mu(aq),\mu(b)\}.
\]
Since $0\leq \min\{\mu(aq),\mu(b)\}$ and $\mu(q)=0$, we have $a,b\in\VR_K$ and we are done.

Now consider an integer $n>1$ and assume that for every $f\in K[x]\setminus(g)$, if $\deg(f)<n$, then there exist
$\lambda_1,\ldots \lambda_s\in\N^\textbf{Q}$ and $a_1,\ldots,a_s\in\VR_K$ such that
\[
f=\sum_{\ell=1}^s a_\ell\textbf{Q}^{\lambda_\ell}
\]
and (\ref{eq:minimalvalue}) holds.

Take $f\in K[x]\setminus(g)$ with $\deg(f)=n$ and $\mu(f)\geq 0$. By our assumption on $\textbf{Q}$, there exists
$q\in \textbf{Q}$ such that $\deg(q)\leq\deg(f)$ and $\mu(f)=\mu_q(f)$. This means that
\[
f=f_0+f_1q+\ldots+f_rq^r\mbox{ with }\deg(f_l)<\deg(q)\mbox{ for every }l, 0\leq l\leq r,
\]
and
\[
0\leq \mu(f)=\min\left\{\mu\left(f_lq^l\right)\right\}\leq \mu(f_l),\mbox{ for every }l, 1\leq l\leq r,
\]
because $\mu(q)=0$. Since $\deg(f_l)<\deg(q)\leq\deg(f)=n$, there exist
\[
\lambda_{0,1},\ldots,\lambda_{0,s_0},\ldots,\lambda_{r,1},\ldots,\lambda_{r,s_r}\in\N^{\textbf{Q}}
\]
and $a_{0,1},\ldots, a_{r,s_r}\in \VR_K$ such that
\[
f_l=\sum_{\ell=1}^{s_l}a_{l,\ell}\textbf{Q}^{\lambda_{l,\ell}}
\]
and $\mu(f_l)=\min\limits_{1\leq \ell \leq s_l}v(a_{l,\ell})$ for every $l$, $0\leq l\leq r$. This implies that
\[
f=\sum_{l=0}^r\left(\sum_{\ell=1}^{s_l}a_{l,\ell}\textbf{Q}^{\lambda_{l,\ell}}\right)q^l=
\sum_{l=0}^r\sum_{\ell=1}^{s_l}a_{l,\ell}\textbf{Q}^{\lambda'_{l,\ell}},
\]
where
\begin{displaymath}
\lambda'_{l,\ell}\left(q'\right)=\left\{
\begin{array}{ll}
\lambda_{l,\ell}(q)+l&\mbox{ if }q=q'\\
\lambda_{l,\ell}(q)&\mbox{ if }q\neq q'
\end{array}
\right.
\end{displaymath}
and $\mu(f)=\min\limits_{\substack{0\leq l\leq r\\1\leq \ell \leq s_l}}v(a_{l,\ell})$.
This concludes our proof.
\end{proof}

\subsection{The algebraic case}
Let $(L|K,v)$ be a simple algebraic extensions of valued fields. Write $L=K(\eta)$ and let $\nu$ be the valuation of $K[x]$ with non-trivial support defined by $v$ and $\eta$:
\[
\nu(f):=v(f(\eta)).
\]

For $\textbf{Q}\subseteq K[x]$ we denote by $\textbf{Q}(\eta):=\{q(\eta)\mid q\in \textbf{Q}\}$.
\begin{Prop}\label{Corfinitelgenalg}
Assume that $e(L|K,v)=1$. For any complete set of polynomials $\textbf{Q}$ for $\nu$ there exists a $K$-proportional complete set
$\tilde{\textbf{Q}}$ such that
\begin{equation}\label{formulagebfvalr}
\VR_L=\VR_K\left[\tilde{\textbf{Q}}(\eta)\right].
\end{equation}
\end{Prop}
\begin{proof}
By Lemma \ref{propostiset} there exists a set $\tilde{\textbf Q}$ which is $K$-proportional such that
\[
\nu(q)=0\mbox{ for every }q\in \tilde{\textbf{Q}}\setminus(g).
\]
By Lemma \ref{lemaparasetpol}, $\tilde{\textbf Q}$ is a also a complete set for $\nu$.

By definition, the right hand side of \eqref{formulagebfvalr} is contained in the left hand side. For every $b\in L$ there exists a polynomial $f(x)\in K[x]$ (with $\deg(f)<\deg(g)$) such that $b=f(\eta)$. If $b\in \VR_L$, then $0\leq\nu(f(x))$. By Theorem \ref{Theoremkeypol}, there exist  $a_1,\ldots,a_r\in \VR_K$ and $\lambda_1,\ldots,\lambda_r\in\N^{\tilde{\textbf{Q}}}$ such that
\[
b=f(\eta)=\sum_{\ell=1}^ra_\ell\tilde{\textbf{Q}}(\eta)^{\lambda_\ell}\in \VR_K\left[\tilde{\textbf{Q}}(\eta)\right].
\]
This concludes the proof.
\end{proof}

\section{A characterization of $\Omega$}
We consider a valued field extension $(L|K,v)$ as before. Namely, $L=K(\eta)$ for some algebraic $\eta\in L$, $g$ is the minimal polynomial of $\eta$ over $K$ and $\nu$ is the valuation of $K[x]$ with non-trivial support defined by $v$ and $\eta$. In this case, $g$ is a key polynomial for $\nu$ with
\[
\delta(g)=\infty>\delta(Q)\mbox{ for every key polynomial } Q\mbox{ for }\nu\mbox{ such that }Q\neq g. 
\]

Take a complete sequence of key polynomials $\textbf{Q}\cup\{g\}$ for $\nu$, where $\textbf{Q}=\{Q_i\}_{i\in I}$. The initial element of $I$ will be denoted by $0$. Assume that $(L|K,v)$ is unramified. For every $i\in I$ choose $a_i\in K$ such that $\nu(Q_i)=v(a_i)$ and set $\tilde Q_i=Q_i/a_i$. We set $Q_0=x$. Replacing $x$ by $\frac x{a_0}$, we may assume that $\nu(x)=0$, in other words, $Q_0=\tilde Q_0=x$. Also, we set $I_{\infty}=I\cup\{\infty\}$ and $Q_\infty=\tilde Q_\infty=g$.

\begin{Def}
In the case described above, we will say that $\{Q_i(\eta)\}_{i\in I}$ is a complete sequence of key polynomials for $(L|K,v)$.
\end{Def}
\medskip

For an element $f\in K[x]$ we denote by $f'$ its formal derivative with respect to $x$. Let $\Gamma$ denote the value group of $v$ and fix a sequence of key polynomials $\{Q_i(\eta)\}_{i\in I}$  for $(L|K,v)$. For each $i\in I$ consider the following values in $\Gamma_\infty$: 
\[
\alpha_i=\nu(Q_i')-\nu(Q_i)\mbox{ and }\beta_i=\nu(g')-\nu_i(g).
\]
Here, we write $\nu_i(g)=\nu_{Q_i}(g)$ for simplicity.

Let $\alpha$ and $\beta$ be the smallest final segments of $\Gamma_\infty$ containing
\[
\{\alpha_i\mid i\in I\}\mbox{ and }\{\beta_i\mid i\in I\},
\]
respectively.

Let $\Omega$ denote the module of K\"ahler differentials for the extension $\VR_L|\VR_K$. It is not hard to show that the $\VR_L$-module homomorphism
\begin{equation}
\Omega\lra I_\alpha/I_\beta\label{eq:themapPsi}
\end{equation}
sending $d\tilde Q_i$ to $Q'_i/a_i$, $i\in I$, is well defined and surjective (below we will prove it in the special cases needed for this paper).

\begin{Que}\label{conjectprin}
Under what conditions is the map \eqref{eq:themapPsi} an isomorphism?
\end{Que}

\begin{Obs}
In full generality, some conditions are definitely needed for \eqref{eq:themapPsi} to be an isomorphism; it is not one unconditionally. The answer depends on the structure of the possible sequences of key polynomials for $(L|K,v)$.
\end{Obs}

\subsection{Consequences of the isomorphism $\Omega\cong I_\alpha/I_\beta$}\label{mainthm}
In this section we assume that we are in a situation where \eqref{eq:themapPsi} is an isomorphism. We will show that this is the case for unramified pure extensions, for instance (see Section \ref{pureexteexts} below).

\begin{Prop}\label{propsannihilator}
The annihilator of $\Omega$ is $I_\gamma\cap \VR_L$ where
\[
\gamma=\{\beta'\in \Gamma\mid \beta'+\alpha\subseteq \beta\}.
\]
\end{Prop}
\begin{proof}
Indeed, since two isomorphic $\VR_L$-modules have the same annihilator, we deduce that
\[
{\rm ann}(\Omega)={\rm ann}\left(I_\alpha/I_\beta\right).
\]
Now an element $a\in \VR_L$ annihilates $I_\alpha/I_\beta$ if and only if $va+\alpha\subseteq \beta$. Clearly, this happens if and only if $va\in \gamma$.
\end{proof}

\begin{Def}
The {\bf invariance subgroup} of a final segment $S$ of $\Gamma$ is defined as
\[
H(S):=\{\beta'\in \Gamma\mid \beta'+S=S\}.
\]
\end{Def}
\begin{Obs}
If $\alpha\neq \beta$, then $H(\alpha)<\gamma$. 
\end{Obs}
Another important consequence of the isomorphism in \eqref{eq:themapPsi} is about the finite generation of $\Omega$.
\begin{Cor}
If $\Omega$ is finitely generated as $\VR_L$-module, then it is generated by a single element.
\end{Cor}
\begin{proof}
Assume that $I_\alpha/I_\beta$ is finitely generated, say, by
\[
\{a_1+I_\beta,\ldots,a_n+I_\beta\}\mbox{ with }a_1,\ldots,a_n\in I_\alpha.
\]
Suppose, without loss of generality, that $v(a_1)\leq v(a_\ell)$ for every $\ell$, $1< \ell\leq n$. Then for every such $\ell$ we have
\[
a_\ell+I_\beta=\frac{a_\ell}{a_1}\left (a_1+I_\beta\right).
\]
Hence, $I_\alpha/I_\beta$ is generated by $\alpha_1+I_\beta$.  
\end{proof}

The above result can be interpreted as follows.

\begin{Prop}\label{finitelygenerated}
Assume that $\Omega\neq(0)$. Then $\Omega$ is finitely generated as an $\VR_L$-module if and only if $\alpha$ has a smallest element.
\end{Prop}
We can use the isomorphism \eqref{eq:themapPsi} to answer the question of when $\Omega$ is finitely presented.
\begin{Cor}\label{finitelypresentes}
Assume that $\Omega\neq(0)$. Then $\Omega$ is finitely presented as an $\VR_L$-module if and only if each of $\alpha$ and $\beta$ has a smallest element.
\end{Cor}
\begin{proof}
If $\alpha$ does not have a smallest element (and $\Omega\neq(0)$), then $I_\alpha/I_\beta$ is not finitely generated, so it is not finitely presented.

Suppose that $\alpha$ has a smallest element, say $\alpha_0=va$ for some $a\in L$. The $\VR_L$-relations on $\{a+I_\beta\}$ are of the form
\[
b\left(a+I_\beta\right)=0\mbox{ for some }b\in \VR_L.
\]
This happens if and only if $v(ab)\in\beta$.

Take finitely many $b_1,\ldots,b_r\in \VR_L$ such that $v(ab_\ell)\in \beta$ for every $\ell$, $1\leq \ell\leq r$. Suppose, without loss of generality, that $v(b_1)\leq v(b_\ell)$ for every $\ell$, $1\leq \ell\leq r$. Then all the relations
\[
b_1\left(a+I_\beta\right)=0,\ldots,b_r\left(a+I_\beta\right)=0,
\]
are generated by
\[
b_1\left(a+I_\beta\right)=0.
\]
In particular, $\beta$ has a smallest element if and only if $I_\alpha/I_\beta$ is finitely presented.
\end{proof}

\section{Pure extensions}\label{pureexteexts}
The main goal of section is to show that \eqref{eq:themapPsi} is an isomorphism for unramified pure extensions.
\begin{Def}\label{definintiaopure}
Let $(L|K,v)$ be a simple immediate extension of valued fields and take $\eta\in L$ such that $L=K(\eta)$. We say that the extension $(L|K,v)$ is {\bf pure immediate} in $\eta$ if for every $f\in K[x]$, $\deg(f)<[L:K]$, there exists $c\in K$ such that for every $b\in K$ we have
\begin{equation}\label{fiexepolyn}
v(\eta-b)\geq v(\eta-c)\Lra v(f(b))=v(f(c)).
\end{equation}
\end{Def}
\begin{Def}
When condition \eqref{fiexepolyn} is satisfied for some $f$, we say that $v(f(c))$ is the fixed value of $f$.
\end{Def}
For $b\in L$ we denote
\[
v(b-K):=\{v(b-c)\mid c\in K\}.
\]
\begin{Lema}\label{Lemapurevspol}
If the extension $(L|K,v)$ is pure immediate in $\eta$, then $v(\eta-K)$ is a complete set of key polynomials for $(L|K,v)$.
\end{Lema}
\begin{proof}
It follows from \cite[Corollary 3.4]{NS2018}.
\end{proof}

\begin{Obs} The above lemma provides a useful characterization of pure immediate extensions: an immediate extension is pure immediate if and only if all of the key polynomials for $(L|K,v)$ are linear. Below we will partially extend this characterization to pure extensions that are not necessarily immediate.
\end{Obs}

\begin{Cor}\label{corfingenera}
Assume that $(L|K,v)$ is pure immediate in $\eta$ and take $\{c_i\}_{i\in I}\subset K$  such that $\{v(\eta-c_i)\}_{i\in I}$ is cofinal in $v(\eta-K)$. For each $i\in I$ take $a_i\in K$ such that $v(\eta-c_i)=v(a_i)$. Then
\[
\VR_L=\VR_K\left[\left.\frac{\eta-c_i}{a_i}\ \right|\ i\in I\right].
\] 
\end{Cor}
\begin{proof}
Follows from Proposition \ref{Corfinitelgenalg}.
\end{proof}
\begin{Def}
The extension $(L|K,v)$ is called \textbf{unibranched} if $v$ is the unique extension of $v|_K$ to $L$.
\end{Def}
The main interest in \cite{CKR} is to study pure extensions when $(L|K,v)$ is algebraic and unibranched (hence $d(L|K,v)=[L:K]$). In view of Lemma \ref{Lemapurevspol}, we extend the definition of pure extensions to the case of algebraic extensions that are not necessarily immediate. 
\begin{Def}
Assume that $(L|K,v)$ is a simple algebraic extension of valued fields. Then we say that:
\begin{description}
\item[(i)] $(L|K,v)$ has \textbf{pure defect} if it is pure immediate and unibranched (hence $d(L|K,v)=[L:K]$). 

\item[(ii)] $(L|K,v)$ is \textbf{branched pure} if it is pure immediate but not unibranched.

\item[(iii)] $(L|K,v)$ is \textbf{purely ramified} if
\[
vL/vK\mbox{ is cyclic and }[L:K]=(vL:vK).
\]

\item[(iv)] $(L|K,v)$ is \textbf{purely inertial} if
\[
Lv|Kv\mbox{ is simple and }[L:K]=[Lv:Kv].
\]
\end{description}
\end{Def}

From now on, when referring to pure extensions we will mean any extension as one appearing in the definition above.
\begin{Prop}\label{proposgenepure}
For any extension of valued fields of the forms \textbf{(i)} to \textbf{(iv)} above there exists a generator $\eta$ of $L|K$ such that the set $v(\eta-K)$ is a complete set of key polynomials for $(L|K,v)$.
\end{Prop}
\begin{proof}
The cases \textbf{(i)} and \textbf{(ii)} follow from Lemma \ref{Lemapurevspol}. In the case \textbf{(iii)} take $\eta\in L$ such that $v\eta$ generates $vL$ over $vK$. In the case \textbf{(iv)} take $\eta\in \VR_L$ such that $\eta v$ generates $Lv$ over $Kv$. In both cases, $\eta$ generates $L$ over $K$. Moreover, if $\nu$ is the valuation of $K[x]$ induced by $v$ and $\eta$, then it is easy to show that for every $f\in K[x]$, $\deg(f)<[L:K]$, we have
\[
\nu(f)=\nu_{x}(f).
\]
\end{proof}
In the case of purely inertial or purely ramified as above, we say that $(L|K,v)$ is pure in $\eta$.

In what follows we present the proof of the isomorphism
\begin{equation}
\Omega\cong I_\alpha/I_\beta\label{eq:Omega=IalphaoverIbeta}
\end{equation}
for unramified pure extensions. This proof is a simplified version of the general proof of this result which will be presented in future work.

\subsection{Proof of \eqref{eq:Omega=IalphaoverIbeta} for pure defect, branched pure or purely inertial extensions}\label{proofofconjes}
Suppose that $(L|K,v)$ is a pure defect, branched pure or purely inertial extension in $\eta$. Let $n=[L:K]$. Take a sequence
\[
\textbf Q=\{Q_i\}_{i\in I}\mbox{ where } Q_i=x-c_i\mbox{ for some }c_i\in K,
\]
such that $\{v(\eta-c_i)\}_{i\in I}$ is cofinal in $v(\eta-K)$.
\begin{Obs}
In the purely inertial case, we can choose $\textbf{Q}$ to be a one-element set.
\end{Obs}
Consider the corresponding $\VR_L$-submodules $I_\alpha$ and $I_\beta$ of $L$.
\begin{Teo}\label{Trheconjecaseoure}
We have the isomorpihsm \eqref{eq:Omega=IalphaoverIbeta} of $\VR_L$-modules.
\end{Teo}
\begin{proof}
It follows from Proposition \ref{proposgenepure} that $\{\eta-c_i\}_{i\in I}$ is a complete sequence of key polynomials for $(L|K,v)$. For each $i\in I$ we have
\[
\alpha_i=\nu(Q_i')-\nu(Q_i)=v1-v(\eta-c_i).
\]
Consequently,
\begin{equation}\label{formulaalpha}
\alpha=-v(\eta-K).
\end{equation}
For every $i\in I$ take $a_i,d_i\in K$ such that
\[
v(a_i)=v(\eta-c_i)\mbox{ and }v(d_i)=\nu_i(g).
\]
We may assume that $v\eta=0$ and $Q_0(\eta)=\eta$. Set $\tilde Q_i=\frac{x-c_i}{a_i}$. Take a set
\[
\textbf X=\{X_i\mid i\in I\}
\]
of independent variables and consider the homomorphism
\[
\VR_K[X_i\mid i\in I]\lra \VR_L
\]
defined by $\textbf{X}\mapsto\tilde{\textbf{Q}}(\eta)$. By Corollary \ref{corfingenera}, this homomorphism is surjective. Denote by $\mathcal I$ its kernel and take a set of generators $\{f_k\}_{k\in J}$ of $\mathcal I$.  We can identify
\[
\Omega=\frac{\bigoplus\limits_{i\in I}\VR_LdX_i}{\mathcal J}
\]
where
\begin{equation}
\mathcal J=(df_k\mid k\in J)\label{eq:generatorsofJ}
\end{equation}
(see \cite{Mat}, Chapter 10, Section 26, (26.E), Example 1 on pp. 183--184 and (26.I), Theorem 58 on pp. 187--188 --- the second fundamental exact sequence --- where we take $A=\VR_K[\bf X]$, $B=\VR_L$ and $\mathfrak m=\mathcal I$).

For  $i,j\in I$, the $K$-relations between $\tilde{Q}_j$ and $\tilde{Q_i}$ are generated by
\[
\tilde{Q}_j=\frac{a_i}{a_j}\tilde{Q}_i+\frac{c_i-c_j}{a_j}.
\]
Hence, all the $\VR_K$-relations between them are generated by
\begin{equation}
b_{ji}\tilde{Q}_j=\tilde{Q}_i+\frac{c_i-c_j}{a_i}.\label{eq:OKrelatgen}
\end{equation}
where $b_{ji}=a_j/a_i$ and $j>i$. All of these relations are generated (over $K$, but not over
$\VR_K$) by
\begin{equation}\label{relatgen}
a_i\tilde{Q_i}=\tilde{Q}_0+c_i\mbox{ for }i\in I.
\end{equation}
Since $v(a_i)=\nu(x-c_i)$ and the $(x-c_i)$-expansion of $g$ is
\[
g(x)=(x-c_i)^n+\ldots+g'(c_i)(x-c_i)+g(c_i)
\]
we see that
\[
d_i=\nu_i(g)=\min\{v(a_i^n),\ldots, v(a_ig'(c_i)),v(g(c_i))\}.
\]
This gives the $K$-relation 
\[
0=g(\eta)=a_i^n\tilde{Q}_i(\eta)^n+\ldots+a_ig'(c_i)\tilde{Q}_i(\eta)+g(c_i).
\]
Hence, the ideal $\mathcal I$ is generated by the elements
\begin{equation}
b_{ji}X_j-X_i-\frac{c_i-c_j}{a_i},j>i.\label{eq:OKrelatgenX}
\end{equation}
and the $\VR_K$-relations of the form
\[
\frac{a_i^n}{d_i} X_i^n+\ldots+\frac{a_ig'(c_i)}{d_i} X_i+\frac{g(c_i)}{d_i}=\frac{g(a_iX_i+c_i)}{d_i}.
\]
Let $g_i:=\frac{g(a_iX_i+c_i)}{d_i}\in\VR_K[X_i]\subset\VR_K[\bf X]$. Then the module $\mathcal J\subset\bigoplus\limits_{i\in I}\VR_LdX_i$ is generated by the elements of the form
\begin{equation}\label{reqrelationomega}
b_{ji}dX_{ j}-dX_{ i},j>i\mbox{ for }i,j\in I
\end{equation}
and
\begin{equation}
\frac{\partial g_i}{\partial X_i}dX_i\mbox{ for }i\in I.\label{reqrelationomega2}
\end{equation}
\begin{Obs} In $\bigoplus\limits_{i\in I}LdX_i$, the element $\frac{\partial g_i}{\partial X_i}dX_i$ is congruent to
$\frac{g'(\eta)}{d_i}dX_0$ modulo the relations \eqref{reqrelationomega}.
\end{Obs}
Consider the map
\begin{displaymath}
\begin{array}{ccccc}
\Psi&:&\Omega&\lra &I_\alpha/I_\beta\\
&&\displaystyle\sum_{l=1}^s b_ldX_{i_l}+\mathcal J&\longmapsto&\displaystyle\sum_{l=1}^s \frac{b_l}{a_{i_l}}+I_\beta
\end{array}.
\end{displaymath}
This can be seen as the map that sends $dX_0+\mathcal J$ to $1+I_\beta$ and is extended to $\Omega$ by the relations $a_idX_i=dX_0$.  The fact that the relations in \eqref{reqrelationomega} and \eqref{reqrelationomega2} generate $\mathcal J$ guarantees that $\Psi$ is well-defined and injective:
\[
\sum_{l=1}^s b_ldX_{i_l}\in \mathcal J
\]
if and only if
\[
\sum_{l=1}^s \frac{b_l}{a_{i_l}}=a\frac{g'(\eta)}{d_i}\mbox{ for some }a\in \VR_L\mbox{ and }i\in I,
\]
if and only if
\[
\sum_{l=1}^s \frac{b_l}{a_{i_l}}\in I_\beta.
\]
Since $\Psi$ is defined on the generators of $\Omega$ and extended by $\VR_L$-linearity, it is automatically a homomorphism of
$\VR_L$-modules. The surjectivity of $\Psi$ follows by construction: if $b\in I_\alpha$, then there exists $i\in I$ such that
$vb\geq -v(a_i)$. Hence
\[
\Psi\left(ba_idX_i+\mathcal J\right)=\frac{ba_i}{a_i}+I_\beta=b+I_\beta.
\]
This concludes the proof of Theorem \ref{Trheconjecaseoure}.
\end{proof}
\subsection{Computation of $\Omega$ for pure defect and branched pure extensions}\label{sectionimportdeffe}
Assume that $(L|K,v)$ is a pure defect extension in $\eta$ and consider a sequence $\textbf Q$ as before. For $i$ large enough, it follows from \cite[Proposition 4.1]{Netal} that $\nu_i(g)=n\nu(Q_i)$. Hence, for $i$ large enough, we have
\[
\beta_i=v(g'(\eta))-\nu_i(g)=v(g'(\eta))-n\cdot v(\eta-c_i).
\]
Consequently,
\begin{equation}\label{formulabeta}
\beta=v(g'(\eta))-n\cdot v(\eta-K).
\end{equation}
Combining this with \eqref{formulaalpha}, we can say precisely what the annihilator of $\Omega$ is:
\begin{equation}\label{formulaannilhi}
{\rm ann}(\Omega)=\left\{b\in \VR_L\ \left|\ v\left(\frac{b}{g'(\eta)}\right)-v(\eta-K)\subseteq -n\cdot v(\eta-K)\right.\right\}. 
\end{equation}
For instance, if $\mbox{rk}\ v=1$, then we may assume that $vL\subseteq \R$ and take
\[
\rho:=\sup v(\eta-K)\in \R.
\]
Then, ${\rm ann}(\Omega)$ can be written as
\[
{\rm ann}(\Omega)=\left\{b\in \VR_L\mid vb\geq (1-n)\rho+v(g'(\eta))\right\}. 
\]

In the case of a branched pure extension, we obtain a slightly different formula for $\beta$ (for $\alpha$ we have the same formula \eqref{formulaalpha}). Namely, by the defect formula (\cite[Theorem 6.14]{NN20221}) we deduce that for large enough $i\in I$ we have
\[
\nu_i(g)=\nu\left(\partial_dg(a_i)(x- c_i)^d\right)=\nu(\partial_dg(a_i))+dv(\eta-c_i),
\]
where $d:=d(L|K,v)$. If $\beta_d$ is the fixed value of $\partial_dg$, then
\begin{equation}\label{purelebranchedequa}
\beta=v(g'(\eta))-\beta_d-d\cdot v(\eta-K).
\end{equation}
\begin{Obs} In the case of a pure defect extension, we have $n=d(L|K,v)=d$, $\partial_dg=1$ (as $g$ is monic) and
$\beta_d=0$. Thus formula \eqref{formulabeta} can be viewed as a special case of \eqref{purelebranchedequa} when the number of distinct extensions of $v|_K$ to $L$ is equal to one.
\end{Obs}

\subsection{About \cite[Proposition 4.1]{CKR}}
In \cite{CKR} an alternative characterization of $\Omega$ is presented. We briefly compare this approach with ours.
\begin{Prop}\cite[Proposition 4.1]{CKR}\label{proposckr}
Let $L|K$ be an algebraic field extension and suppose that $A$ is a normal domain with quotient field $K$ and $B$ is a domain with quotient field $L$ such that $A\subset B$ is an integral extension. Suppose that there exist $s_\alpha \in B$, which are indexed by a totally ordered set $S$, such that $A[s_\alpha]\subset A[s_\beta]$ if $\alpha<\beta$ and
\[
\bigcup_{\alpha\in S} A[s_\alpha]=B.
\]
Further, suppose that there exist $r_\alpha,r_\beta\in A$ such that $r_\beta\mid r_\alpha$ if $\alpha\leq \beta$ and for $\alpha\leq \beta$, there exist $c_{\alpha,\beta}\in A$ and expressions
\[
s_\alpha=\frac{r_\alpha}{r_\beta}s_\beta+c_{\alpha,\beta}.
\]
Let $h_\alpha$ be the minimal polynomial of $s_\alpha$ over $K$. Let $U$ and $V$ be the $B$-ideals
\[
U = (r_\alpha \mid \alpha\in S)\mbox{ and }V=(h_\alpha'(s_\alpha)\mid \alpha \in S).
\]
Then we have a $B$-module isomorphism
\[
\Omega_{B|A}\simeq U/UV.
\]
\end{Prop}

We now explain how to describe the elements $r_\alpha$, $s_\alpha$ and $c_{\alpha\beta}$ of Proposition \ref{proposckr} using our terminology and notation.
\begin{Obs}
It is well-known that for an algebraic extension $(L|K,v)$ the valuation ring $\VR_K$ is normal and its integral closure in $L$ is the intersection of all the valuation rings of $L$ dominating $\VR_K$. Hence, for the extension $\VR_L|\VR_K$ to be integral, it is necessary for $(L|K,v)$ to be unibranched. 
\end{Obs}

Assume that $(L|K,v)$ is a pure defect extension and consider the sequence of key polynomials $\{\eta-c_i\}$ and elements $a_i$ as before. Take $\tilde r\in \VR_L$ such that
\[
v\tilde r>v(\eta-K).
\]
For each $i\in I$, write
\[
r_i=\tilde r a_i^{-1}\mbox{ and }s_i=\frac{\eta-c_i}{a_i}.
\]
If $i<j$, then
\[
s_i=\frac{(\eta-c_i)}{a_i}=\frac{r_i}{r_j}\cdot s_j+\frac{(c_j-c_i)}{a_i}\in \VR_K[s_j].
\]
Since $v(a_i)=v(\eta-c_i)<v(\eta-c_j)=v(a_j)$ we deduce that $\tilde ra_i^{-1}=r_i$ is a multiple of $r_j=\tilde ra_j^{-1}$ by an element of $\VR_L$. In particular, the hypotheses of \cite[Proposition 4.1]{CKR} are satisfied and we deduce that
\[
\Omega\simeq U/UV.
\]
Let us compute $U$ and $V$ and compare with our results above. For each $i\in I$, we have
\[
v(r_i)=v(\tilde ra_i^{-1})=v\tilde r-v(a_i)
\]
and consequently
\begin{equation}
U=(r_i\mid i\in I)=\{b\in \VR_L\mid vb>v\tilde r-v(a_i)\mbox{ for some }i\in I\}= I_{\alpha+v\tilde r}.
\end{equation}

Let us calculate $V$. For each $i$, we have $\eta=(\eta-c_i)+c_i$. Hence
\begin{equation}
0=g(\eta)=g((\eta-c_i)+c_i)=g(c_i)+g'(c_i)(\eta-c_i)+\ldots+(\eta-c_i)^n.\label{eq:minpolynQi}
\end{equation}
\begin{Obs}\label{dividebyain} Since $\deg\ g'<\deg\ g$, for $i$ sufficiently large the sequence
\[
\{v\left(g'(c_i)\right)\}_{i\in I}
\]
stabilizes and its stable value is equal to $v(g'(\eta))$. Of course, this includes the case $g'=0$.
\end{Obs}
Since $s_i=\frac{\eta-c_i}{a_i}$, \eqref{eq:minpolynQi} implies that
\[
0=\frac{g(c_i)}{a_i^n}+\frac{g'(c_i)}{a_i^{n-1}}s_i+\ldots+s_i^n.
\]
Write
\begin{equation}
h_i(x)=\frac{g(a_ix+c_i)}{a_i^n}=\frac{g(c_i)}{a_i^n}+\frac{g'(c_i)}{a_i^{n-1}}x+\ldots+x^n\in K[x].\label{eq:hiintermsofg}
\end{equation}
Since $g(x)$ is irreducible over $K$, so is $h_i(x)$. Thus $h_i(x)$ is the minimal polynomial of $s_i$. 

In view of the expression \eqref{eq:hiintermsofg} for $h_i$ in terms of $g$ and the chain rule for differentiation, we have, for all $i$,
\[
h_i'(x)=\frac{g'(a_ix+c_i)}{a_i^{n-1}},
\]
so
\[
h_i'(s_i)=\frac{g'(a_is_i+c_i)}{a_i^{n-1}}=\frac{g'(\eta)}{a_i^{n-1}}.
\]

Therefore,
\begin{displaymath}
\begin{array}{rcl}
V&=&\displaystyle\left(v\left(\frac{g'(\eta)}{a_i^{n-1}}\right)\ \left| \right.\ i\in I\right)\\[8pt]
&=&\{b\in \VR_L\mid \ \exists i\in I\mbox{ such that }v(b)\geq v(g'(\eta))+(1-n)va_i\}\\[8pt]
&=&\displaystyle I_{v(g'(\eta))+(1-n)v(\eta-K)}.

\end{array}
\end{displaymath}
\section{Minimal key polynomials}\label{minkeypolyno}
Let $\nu$ be a valuation of $K[x]$ and for $m\in\N$ consider the set
\[
\Psi_m=\{Q\in K[x]\mid \deg(Q)=m\mbox{ and }Q\mbox{ is a key polynomials for }\nu\}. 
\]
If $\Psi_m\neq \emptyset$ and $F$ is a key polynomial of smallest degree strictly larger than $m$ (we denote $\deg\ F$ by
$m_+$, so $F\in\Psi_{m_+}$), then we say that $F$ is a \textbf{key polynomial for $\Psi_m$}. If $\nu(\Psi_m)$ does not have a maximum, then every key polynomial for $\Psi_m$ will be called a \textbf{limit key polynomial for $\Psi_m$}.

For two polynomials $f,q\in K[x]$, $\deg(q)>0$, denote by
\[
f=f_{r,q}q^r+\ldots+f_{0,q}
\]
the $q$-expansion of $f$.

For each $m\in \N$ such that $\Psi_m\neq \emptyset$ we will choose a subset $\textbf{Q}_m$ as follows. If $\nu(\Psi_m)$ has a maximum, then set $\textbf{Q}_m=\{Q_m\}$ where $Q_m\in \Psi_m$ has this maximum value. If $\nu(\Psi_m)$ does not have a maximum, we choose $\textbf{Q}_m$ to be any well-ordered cofinal subset of $\Psi_m$.

\begin{Def}
A key polynomial $F$ for $\Psi_m$ is said to be \textbf{minimal} if there exists $Q\in\textbf{Q}_m$ and a subset $B_m$ of $\N$ such that for every $R\in \textbf{Q}_m$, with
\[
\nu(R)\geq \nu(Q),
\]
the polynomial
\[
F_{0,R}+\sum_{\ell\in B_m}F_{\ell,R}R^\ell
\]
is a key polynomial for $\Psi_m$ but for every $s\in B_m$ there exists a cofinal subset of $\textbf{Q}_m$ such for every $R$ in it, the polynomial
\[
F_{0,R}+\sum_{\ell\in B_m\setminus\{s\}}F_{\ell,R}R^\ell
\] 
is not a key polynomial for $\Psi_m$.
\end{Def}
One can show that the set $B_m$ above does not depend on $F$ or $Q\in\Psi_m$.
\begin{Obs}
In the case when $\Psi_m$ contains a maximal element, the above definition becomes much simpler. Since the only possibility for $Q$ is the unique element of $\mathbf Q_m$, the polynomial $F$ is a minimal key polynomial for $\Psi_m$ if
\[
F_{0,Q}+\sum_{\ell\in B_m}F_{\ell,Q}Q^\ell
\]
is a key polynomial for $\Psi_m$ but for every $s\in B_m$ the polynomial
\[
F_{0,Q}+\sum_{\ell\in B_m\setminus\{s\}}F_{\ell,Q}Q^\ell
\] 
is not a key polynomial for $\Psi_m$.
\end{Obs}

For a pure extension $(L|K,v)$ of degree $n$ and a generator $\eta$ for it, consider the valuation $\nu$ on $K[x]$  defined by $v$ and $\eta$. The only natural numbers $m$ for which the corresponding set $\Psi_m$ is non-empty are $m=1$ and $m=n$. Hence, the only case for which it makes sense to talk about minimal key polynomials is for the set $\Psi_1$. Because of this, for the remaining of this section we will denote $B_1$ simply by $B$.
\subsection{The purely inertial case}
Assume that the extension $(L|K,v)$ is purely inertial. Choosing $\eta$ as in Proposition \ref{proposgenepure} we deduce that $0=v\eta$ is the maximum of $v(\eta-K)$ and $\eta v$ generates $Lv|Kv$.

The following lemma follows immediately from Proposition \ref{Corfinitelgenalg}.
\begin{Lema}
We have $\VR_L=\VR_K[\eta]$.
\end{Lema}

\begin{Prop}\label{porpcaseinertial}
If $(L|K,v)$ is purely inertial, then
\[
\Omega \simeq \VR_L/(g'(\eta)).
\]
\end{Prop}
\begin{proof}
In this case we can apply Theorem \ref{Trheconjecaseoure} for  $\textbf{Q}=\{x\}$. Then we have
\[
\alpha=v1-v\eta=0\mbox{ and }\beta=v(g'(\eta))-\nu_x(g)=v(g'(\eta)).
\]
Hence, the result follows.
\end{proof}

Let
\[
\overline Q(y)=b_0v+b_1vy+\ldots+y^n\in Kv[y]
\]
be the minimal polynomial of $\eta v$ over $Kv$. Then 
\[
Q=b_0+b_1 x+\ldots+ x^n=b_0+\sum_{k\in B}b_kx^k\in K[x]
\]
is a minimal key polynomial for $\textbf{Q}_1$. Also, we deduce that
\[
\nu(g-Q)= \nu_x(g-Q)>\nu_x(g)=0.
\]

\begin{Cor}\label{corlopinertial}
If $(L|K,v)$ is purely inertial, then
\[
\Omega=(0) \Llr Lv|Kv\mbox{ is separable}\Llr p\nmid B.
\]
\end{Cor}
\begin{proof}
By Proposition \ref{porpcaseinertial} we deduce that $\Omega=(0)$ if and only if $v(g'(\eta))=0$. This happens if and only if
$\frac{d\overline{Q}}{dy}\neq 0$. This happens if and only if $\eta v$ is separable over $Kv$. Since
\[
\overline Q(y)=b_0v+b_1vy+\ldots+y^n=b_0v+\sum_{k\in B}b_kv y^k
\]
we deduce that
\[
\frac{d\overline{Q}}{dy}\neq 0\Llr p\nmid B.
\]
\end{proof}
\subsection{A description of $\Omega$ for pure defect and branched pure extensions}
\begin{Prop}\label{Propbunitinhamans}
Assume that $(L/K,v)$ is a pure defect or a branched pure extensions. Then
\[
\Omega=(0)\Llr 1\in B.
\] 
\end{Prop}
\begin{proof}
Set $n=\deg(g)$. Take a cofinal well-ordered (with respect to $\nu$) sequence $\textbf{Q}_1=\{x-a_i\}_{i\in I}$ whose values are cofinal in $v(\eta-K)$ as before.

For each $i\in I$ we have
\[
\alpha_i=\nu\left(\frac{d }{dx}(x-a_i)\right)-\nu(x-a_i)=\nu(1)-\nu(x-a_i)=-\nu(x-a_i).
\]

Since $g$ is a polynomial of smallest degree which is $\textbf{Q}_1$-unstable, there exists $i_0\in I$ such that
\[
\nu(g')=\nu_i(g')=\nu(g'(a_i))\mbox{ for every }i\in I\mbox{ with }i\geq i_0. 
\]
Hence,
\[
\beta_i=\nu(g')-\nu_i(g)\mbox{ for every }i\in I\mbox{ with }i\geq i_0.
\]
In particular, $\{\beta_i\}_{i\in I}$ is ultimately decreasing.

Clearly, $\{\alpha_i\}_{i\in I}$ is decreasing. Since $\{\beta_i\}_{i\in I}$ is ultimately decreasing, there exists a final set
$I_{0}\geq i_0$ of $I$ such that $\{\alpha_i\}_{i\in I_0}$ and $\{\beta_i\}_{i\in I_0}$ are decreasing.

For each $i\in I$, the $(x-a_i)$-expansion of $g$ is given by
\[
g=g(a_i)+g'(a_i)(x-a_i)+\ldots+(x-a_i)^n.
\]
Then the condition $\alpha=\beta$ is equivalent to saying that for every $i\in I_0$, there exists $j\in I_0$ such that
\begin{equation}\label{equacbonia}
\nu(g'(a_i))-\nu_j(g)=\nu(g')-\nu_j(g)=\beta_j<\alpha_i=-\nu(x-a_i).
\end{equation}
This happens if and only if
\[
\nu\left(g'(a_i)(x-a_i)\right)<\nu_j(g).
\]
This shows that $\Omega=(0)$ if and only if $1\in B$.
\end{proof}

\subsection{The branched pure case}
The next result is the equivalent of Corollary \ref{corlopinertial} for branched pure extensions.
\begin{Prop}\label{casebranches}
Assume that $(L|K,v)$ is a branched pure extension. Then
\[
\Omega=(0)\mbox{ if and only if }d(L|K,v)=1.
\]
\end{Prop}
\begin{proof}
In this case we are in the situation of Proposition \ref{Propbunitinhamans}, so it is enough to show that $1\in B$ if and only if $d(L|K,v)=1$. This follows from \cite[Thereom 6.22 and Theorem 6.26]{NN2023}.
\end{proof}

\subsection{A general result}
The main goal of this section is to prove the following.

\begin{Teo}\label{classificatputre}
Assume that $(L|K,v)$ is a pure defect, a branched pure or a purely inertial extension. Then $\Omega=(0)$ if and only if $p\nmid B$.
\end{Teo}
\begin{proof}
If the extension is of pure defect, then it follows from Proposition \ref{Propbunitinhamans} that $\Omega=(0)$ if and only if $1\in B$. On the other hand, it follows from \cite[Theorem 5.10 and Corollary 5.11]{NN2023} that
\[
B\subseteq \{1,p,p^2,\ldots\}.
\]
This proves the result for pure defect extensions.

If the extension is purely inertial, then the result was proved in Corollary \ref{corlopinertial}.

If $(L|K,v)$ is a branched pure extension, then we can apply Proposition \ref{Propbunitinhamans} to obtain that $\Omega=(0)$ if and only if $1\in B$. Since $d\in B$, if $d=1$, then $\Omega=(0)$. On the other hand, it follows from \cite[Theorem 6.26]{NN2023} that if $1\notin B$, then every element of $B$ is a multiple of $p$. This concludes the proof of the theorem.  

\end{proof}

\subsection{The purely ramified case}\label{seclmpureramifi}
We discuss now one particular case where we cannot apply the above results. Assume that the extension $(L|K,v)$ is purely ramified. Choosing $\eta$ as in Proposition \ref{proposgenepure} we deduce that $\gamma=v\eta$ is the maximum of $v(\eta-K)$ and generates $vL$ over $vK$.

Assume, without loss of generality, that $\gamma>0$. Furthermore, let $\Delta$ denote the greatest (in the sense of inclusion) isolated subgroup of $vL$ such that $\Delta<\gamma$.
\begin{Exa}
To illustrate the definition of $\Delta$, we present an example where $\Delta\neq (0)$. Let $(L|K,v)$ be a valued field extension such that
\[
vL=\frac{1}{2}\Z\times_{\rm lex}\R\supset \Z\times_{\rm lex}\R=vK
\]
and $\gamma=\left(\frac 12,0\right)$. Then $\Delta=(0)\times_{\rm lex}\R$. Observe that in this case, $vL_{>0}$ does not have a minimal element, but $\left(\frac{vL}{\Delta}\right)_{>0}$ does.
\end{Exa}

If $\left(\frac{vL}\Delta\right)_{>0}$ has a minimal element, then we will choose $\eta$ in such a way that $\gamma+\Delta=\min\left(\frac{vL}\Delta\right)_{>0}$.

Set $\textbf{Q}_1=\{ x\}$. Let
\[
n=[K(\eta):K].
\]
Since the extension is purely ramified, there exists $b_0\in K$ such that
\[
v\left(\eta^n-b_0\right)>nv\eta=vb_0.
\]
In this case, $Q=x^n-b_0$ is a minimal key polynomial for $\textbf{Q}_1$. In particular, $B=\{n\}$.

\begin{Lema}\label{generationofOL}
Under the above assumptions we have
\begin{equation}
\VR_L=\VR_K\left[\left.\frac{\eta}h\ \right|\ h\in K\mbox{ and }vh<\gamma\right].\label{eq:VRcase1}
\end{equation}
\end{Lema}
\begin{proof}

It is obvious that the right hand side of \eqref{eq:VRcase1} is contained in the left hand side. To prove the opposite inclusion, fix an element $b\in\VR_L$ and write $b=f(\eta)\in K[\eta]$ with
$\deg(f)=m<n$:
\begin{equation}
b=b_0+b_1\eta+\ldots+b_m\eta^m\mbox{ where }b_s=\partial_sf(0)\in K.\label{eq:bisinKeta}
\end{equation}
The elements $s\gamma$, $0\leq s\leq m$, belong to different $vK$-cosets of $vL$. We deduce that
\begin{equation}
0\leq v(b)=\min\limits_{0\leq s\leq m}\{v(b_s)+s\gamma\}.\label{eq:0<vb=min}
\end{equation}

By \eqref{eq:0<vb=min}, we have
\begin{equation}
-v(b_s)<s\gamma\mbox{ for every }s, 1\leq  s\leq m.\label{eq:interval}
\end{equation}
We claim that there exists $h\in K$ such that for all integers $s$, $1\leq s\leq m$, we have 
\begin{equation}
-v(b_s)\leq svh<s\gamma.\label{eq:intervald}
\end{equation}
Denote by $\Delta^+$ the smallest initial segment of $vL$ containing $\Delta$.

If $\left(\frac{vL}\Delta\right)_{>0}$ contains a minimal element then, by the choice of $\gamma$, we have
\[
-v(b_s)\in\Delta^+
\]
for all $s$, so we may take $h$ to be any element of $K$ such that
\[
v(h)\in\Delta\mbox{ and }v(h)\ge\max\limits_{1\le s\le m}\left\{-\frac{v(b_s)}{s}\right\}.
\]

Assume that $\left(\frac{vL}\Delta\right)_{>0}$ does not contain a minimal element. Then every non-empty open interval in
$\frac{vL}\Delta$ contains an element of $vK+\Delta$. Hence, there exists $h\in K$ such that
\begin{equation}
\max_{1\le s\le m}\left\{-\frac{v(b_s)}{s}\right\}+\Delta<vh+\Delta<\gamma+\Delta\label{eq:interval}
\end{equation}
so, in particular,
\begin{equation}
\max_{1\le s\le m}\left\{-\frac{v(b_s)}{s}\right\}<vh<\gamma.\label{eq:intervalbis}
\end{equation}
This completes the proof of the existence of $h\in K$ satisfying \eqref{eq:intervald}. Then
\[
b=b_0+b_1 h\frac\eta h+\ldots+b_mh^m\left(\frac \eta h\right)^m\in\VR_K\left[\frac\eta h\right].
\]
This completes the proof of the lemma.
\end{proof}

\begin{Obs}\label{invariancegroup}
Consider the final segment $-\Delta^+$ of $vL$. Then the invariance subgroup of $-\Delta^+$ is exactly $\Delta$. 
\end{Obs}

Write
\[
g=a_0+a_1x+\ldots+x^n.
\]
\begin{Obs}
For any $\ell$, $1\leq \ell\leq n-1$, we have
\begin{equation}\label{equacoefing}
v\ell+va_\ell-(n-\ell)\gamma>0.
\end{equation}
Indeed, since $v(g(\eta))=\infty$ and $\gamma,2\gamma, (n-1)\gamma, n\gamma$ belong to distinct $vK$-cosets, we deduce that
\[
va_0=n\gamma<\min_{1\leq \ell\leq n-1} \{va_\ell+\ell\gamma\}.
\]
The inequality \eqref{equacoefing} follows.
\end{Obs}
\begin{Prop}\label{purelyramified}
Assume that the extension $(L|K,v)$ is purely ramified.
\begin{description}
\item[(i)] If $\left(\frac{vL}\Delta\right)_{>0}$ contains a minimal element, then
$\Omega\neq (0)$.
\item[(ii)] If $\left(\frac{vL}\Delta\right)_{>0}$ does not contain a minimal element, then $\Omega=(0)$ if and only if there exists
$\ell$, $1\leq \ell\leq n$, such that
\begin{equation}
v\ell+va_\ell-(n-\ell)\gamma\in\Delta.\label{eq:pnmidB1}
\end{equation}
\end{description}
\end{Prop}

\begin{proof}
For $h\in K$ with $vh\le\gamma$, let $J_h$ denote the $\VR_K$-submodule of $K$ defined by
\[
J_h=\{b\in K\ |\ bg(hx)\in\VR_K[x]\}.
\]
It follows from Lemma \ref{generationofOL} that
\begin{equation}\label{eq:Omegaingeneral}
\Omega=\frac{\bigcup\limits_{\substack{h\in K\\vh\le\gamma}}\frac1h\VR_Ldx}{\left(bg'(\eta)dx\ \left|\ b\in\bigcup\limits_{\substack{h\in K\\vh\le\gamma}}J_h\right.\right)}.
\end{equation}
\begin{Obs} We have $J_h=\frac1{h^n}\VR_K$, so \eqref{eq:Omegaingeneral} becomes
\begin{equation}\label{eq:Omegaingeneral1}
\Omega=\frac{\bigcup\limits_{\substack{h\in K\\vh\le\gamma}}\frac1h\VR_Ldx}{\left(bg'(\eta)dx\ \left|\ b\in\bigcup\limits_{\substack{h\in K\\vh\le\gamma}}\frac1{h^n}\VR_K\right.\right)}.
\end{equation}
\end{Obs}
The equality \eqref{eq:Omegaingeneral1} says that $\Omega=(0)$ if and only if the following equality of $\VR_L$-submodules in $L$ holds:
\begin{equation}\label{eq:eqfractideals}
\bigcup\limits_{\substack{h\in K\\vh\le\gamma}}\frac1h\VR_L=\left(bg'(\eta)\ \left|\ b\in\bigcup\limits_{\substack{h\in K\\vh\le\gamma}}\frac1{h^n}\VR_K\right.\right)\VR_L.
\end{equation}
Furthermore, we have

\begin{equation}
v\left(g'(\eta)\right)=\min\limits_{1\leq \ell\leq n}\{v\ell+va_\ell+(\ell-1)\gamma\}.\label{vbg'eta}
\end{equation}
In the rest of the proof of the proposition, we will consider cases \textbf{(i)} and \textbf{(ii)} separately.

Suppose that the hypothesis of \textbf{(i)} is satisfied. In this case the condition $vh\le\gamma$ is equivalent to saying that $vh\in \Delta^+$. Hence, the left hand side of \eqref{eq:eqfractideals} is equal to $I_{-\Delta^+}$. On the other hand, for each $h\in K$ with $vh\le\gamma$ we have $v(J_h)\subset -\Delta^+$. In particular, the right hand side of \eqref{eq:eqfractideals} is contained in
\[
I_{v(g'(\eta))-\Delta^+}.
\]
Since $v(g'(\eta))>\Delta$, by Remark \ref{invariancegroup}, we deduce that
\[
v(g'(\eta))-\Delta^+\subsetneqq -\Delta^+.
\]
Hence, the right hand side of \eqref{eq:eqfractideals} is strictly contained in its left hand side, so $\Omega\ne (0)$, as desired.

Suppose now that the hypothesis of \textbf{(ii)} is satisfied. Taking into account the equality \eqref{vbg'eta}, we see that the $\VR_L$-submodule on the left hand side of \eqref{eq:eqfractideals} is $I_\alpha$, where
\[
\alpha=\{-vh\ |\ h\in K,vh<\gamma\}=vL_{>-\gamma}
\]
and the $\VR_L$-submodule on the right hand side is $I_\beta$, where
\[
\beta=\left\{\left.\min\limits_{1\leq \ell \leq n}\{v\ell+va_\ell+(\ell-1)\gamma\}-nvh\ \right|\ h\in K,vh<\gamma\right\}.
\]
Rewrite $\beta$ as
\begin{displaymath}
\begin{array}{rcl}
\beta&=&\left\{\left.\min\limits_{1\leq \ell\leq n}\{v\ell+va_\ell-(n-\ell)\gamma\}+n(\gamma-vh)-\gamma\ \right|\ h\in K,vh<\gamma\right\}\\[8pt]
&=&\min\limits_{1\leq \ell \leq n}\{v\ell+va_\ell-(n-\ell)\gamma\}+vL_{>-\gamma},
\end{array}
\end{displaymath}
where the last equality holds because $\left(\frac{vL}\Delta\right)_{>0}$ does not contain a minimal element. A similar reasoning as Remark \ref{invariancegroup} shows that the invariance subgroup of $vL_{>-\gamma}$ is $\Delta$. It follows that
\[
\alpha=\beta\Llr \min\limits_{1\leq \ell \leq n}\{v\ell+va_\ell-(n-\ell)\gamma\}\in \Delta.
\]
The latter condition is equivalent to saying that $v\ell+va_\ell-(n-\ell)\gamma\in\Delta$ for some $\ell$, $1\leq \ell\leq n$. This completes the proof of the proposition.
\end{proof}
\begin{Cor}
Assume that $\left(\frac{vL}{\Delta}\right)_{>0}$ does not have a minimum. If $p\nmid B$, then $\Omega=(0)$.
\end{Cor}
\begin{proof}
Since $B=\{n\}$, the condition $p\nmid B$ is the same as $p\nmid n$. Hence, $vn=0$. In particular, the condition \eqref{eq:pnmidB1} is satisfied for $\ell=n$. Consequently, $\Omega=(0)$.
\end{proof}
\begin{Cor}
Assume that $\left(\frac{vL}{\Delta}\right)_{>0}$ does not have a minimum.
If $vp\in \Delta$, then $\Omega=(0)$.
\end{Cor}
\begin{proof}
If $vp\in \Delta$, then $vr\in \Delta$ for every $r\in\N$. Hence, condition \eqref{eq:pnmidB1} is satisfied for $\ell=n$ and consequently, $\Omega=(0)$.
\end{proof}

\section{Artin--Schreier and Kummer extensions}
We apply the results of this paper to the particular cases of Kummer and Artin-Schreier extensions.
\subsection{Artin--Schreier extensions}
Assume that $(L|K,v)$ is an Artin--Schreier extension of degree $p={\rm char}(K)>0$. This means that $L=K(\eta)$ where $\eta$ is a root of an irreducible polynomial over $K$ of the form $g=x^p-x-a$. There are four cases to be considered.

If $d(L|K,v)>1$, then $(L|K,v)$ is a pure defect extension. Applying the results of Section \ref{sectionimportdeffe}, we obtain
\begin{equation}\label{equationalpha}
\alpha=-v(\eta-K)\mbox{ and }\beta=-p\cdot v(\eta-K).
\end{equation}
Hence, ${\rm ann}(\Omega)=I_\gamma\cap \VR_L$, where $\gamma$ is the final segment of $\Gamma$ formed by all the elements
$\beta'$ in $\Gamma$ for which there exists $i\in I$ such that
\[
\beta'+pv(\eta-a_i)> v(\eta-a_j)\mbox{ for every }j\in I.
\]
In the rank one case, we set $\rho=\sup v(\eta-K)$. Then
\[
{\rm ann}(\Omega)=\{b\in \VR_L\mid vb\geq (1-p)\rho\}.
\]
It follows from Proposition \ref{finitelygenerated} that if $\alpha\neq \beta$, then $\Omega$ is not finitely generated (and, consequently, not finitely presented).

If the extension is defectless and not unibranched, then it is branched pure and $d(L|K,v)=1$. By Proposition \ref{casebranches} we deduce that $\Omega=(0)$.

For the rest of this section, assume that $(L|K,v)$ is unibranched and defectless. Then, by the fundamental inequality,
\begin{equation}\label{fundamentinequality}
[L:K]=(vL:vK)[Lv:Kv].
\end{equation}
Since $[L:K]$ is prime, if $Lv\neq Kv$, then $(L|K,v)$ is purely inertial. We can apply Theorem \ref{classificatputre} to obtain that
\[
\Omega=(0)\Llr p\nmid B\Llr Lv|Kv\mbox{ is separable}.
\]
Alternatively, we can present the computation directly from Theorem \ref{Trheconjecaseoure}. Take $\tilde{c}\in K$ with $v\eta=-v(\tilde c)$. In this particular case, by Proposition \ref{Corfinitelgenalg} we have $\VR_L=\VR_K[\tilde c \eta]$. We have
\[
\alpha=v\tilde{c}\mbox{ and }\beta=pv\tilde{c}.
\]
In particular, we deduce that
\[
{\rm ann}(\Omega)=I_{\beta-\alpha}=\{a\in \VR_L\mid va\geq (p-1)v\tilde c\}.
\]
If $v\eta<0$ (so $v\tilde c>0$), then the residue field extension is purely inseparable and we have $\Omega\neq(0)$. If $v\eta=0$, then the residue field extension is separable and $\Omega=(0)$ (in particular, ${\rm ann}(\Omega)=\VR_L$). It follows from Corollary \ref{finitelypresentes} that $\Omega$ is finitely presented.
\begin{Obs}
If the residue field extension is purely inseparable ($v\eta<0$) we consider the ``translation map"
\[
\VR_L\lra I_\alpha/I_\beta, \ a\mapsto a\tilde c+I_\beta. 
\]
Clearly this map is surjective and its kernel is $\tilde c^{p-1}\VR_L$. Hence we deduce \cite[Theorem 4.4]{CK}.
\end{Obs} 

If $vK\neq vL$, then by \eqref{fundamentinequality} $[L:K]=(vL:vK)=p$. In  particular, $(L|K,v)$ is purely ramified. Since $p=(vL:vK)$ we deduce from Theorem \ref{purelyramified} that $\Omega\neq (0)$. Indeed, if $\left(\frac{vL}{\Delta}\right)_{>0}$ has a minimum, then the result follows from Theorem \ref{purelyramified} \textbf{(i)}. Otherwise, observe that
\[
v(a_1)+(n-1)\gamma\geq \gamma>\Delta \mbox{ and }vp=\infty>\Delta
\]
hence the result follows from Theorem \ref{purelyramified} \textbf{(ii)}.
\subsection{Kummer extensions}
Assume that $(L|K,v)$ is a Kummer extension of degree $q$ and let $p={\rm char}(Kv)>0$. We can write $L=K(\eta)$ where $g=x^q-a$ is the minimal polynomial of $\eta$ over $K$.

If $(L|K,v)$ is immediate and unibranched, then it is of pure defect. In particular, $p=q=d(L|K,v)$. We can assume that $v\eta=0$. By the results of Section \ref{sectionimportdeffe} we obtain
\[
\alpha=-v(\eta-K)\mbox{ and }\beta=vp-p\cdot v(\eta-K).
\]
Analogously to the Artin--Schreier case, we deduce that ${\rm ann}(\Omega)=I_\gamma\cap \VR_L$ where $\gamma$ is the final segment formed by all the elements $\beta'$ in $\Gamma$ for which there exists $i\in I$ such that
\[
\beta'-vp+pv(\eta-a_i)>v(\eta-a_j)\mbox{ for every }j\in I.
\]
In rank one, this can be written as
\[
{\rm ann}(\Omega)=\{b\in \VR_L\mid vb\geq vp+(1-p)\rho\}\mbox{ where }\rho=\sup v(\eta-K).
\]
We can see that if $\alpha\neq \beta$, then it follows from Proposition \ref{finitelygenerated} that $\Omega$ is not finitely generated (and hence not finitely presented).

If the extension is branched pure, then  by Proposition \ref{casebranches} we have $\Omega=(0)$ if and only if $d:=d(L|K,v)=1$. Let $\beta_d$ be the fixed value of $\partial_dg$. By \cite{NN2023}, if
\[
vp<r(\beta_d+d\cdot v(\eta-c))\mbox{ for some }r\in \Z \mbox{ and }c\in K,
\]
then $\Omega=(0)$.

For the remainder of this section, assume that $(L|K,v)$ is unibranched and defectless. As in the Artin-Schreier case, if $Lv\neq Kv$, then $(L|K,v)$ is purely inertial. By Theorem \ref{classificatputre} we obtain that
\[
\Omega=(0)\Llr p\nmid B\Llr Lv|Kv\mbox{ is separable}.
\]
More explicitly, taking $\tilde c\in K$ such that $v\eta=-v\tilde c$, we have
\[
\alpha=-v\eta=v\tilde c\mbox{ and }\beta=v(q\eta^{q-1})-qv(\eta)=vq+(q-2)v\tilde c.
\]
Now $\alpha=\beta$ if and only if $v\eta=-v\tilde c=\frac{vq}{q-1}$. If $Lv|Kv$ is not separable, then we have
\[
{\rm ann}(\Omega)=I_{\beta-\alpha}=\{b\in \VR_L\mid vb\geq vq-(q-1)v(\eta)\}.
\]
It follows from Proposition \ref{finitelypresentes} that $\Omega$ is finitely presented.
 
As in the Artin-Schreier case, if $vK\neq vL$, 	then $(L|K,v)$ is purely ramified. As before, if $\left(\frac{vL}{\Delta}\right)_{>0}$ has a minimum, then $\Omega\neq (0)$. Assume that $\left(\frac{vL}{\Delta}\right)_{>0}$ does not have a minimum. If $p\neq q$, then $vq=0$ and by Theorem \ref{purelyramified} \textbf{(ii)} we deduce that $\Omega=(0)$. On the other hand, if $p= q$, then  Theorem \ref{purelyramified} \textbf{(ii)} implies that
\[
\Omega=(0)\Llr vp\in \Delta.
\]
\begin{Obs}
An alternative characterization for $\Omega=(0)$ for a purely ramified Kummer extension is presented in \cite[Theorem 4.7]{CK}. 
\end{Obs}

\end{document}